\documentclass[11pt]{amsart}

\usepackage{amsmath,amsthm,amsfonts,amscd,amssymb,eucal,latexsym}
\usepackage[margin=1.25in,marginparwidth=2.5cm]{geometry}

\usepackage[numbers, sort&compress]{natbib}
\usepackage{graphicx}
\usepackage{vmargin, enumerate}
\usepackage{setspace}
\usepackage{paralist}
\usepackage{caption,subcaption}
\usepackage{wrapfig}
\usepackage[pagebackref=false]{hyperref}

\theoremstyle{plain}
\newtheorem{theorem}{Theorem}[section]
\newtheorem{corollary}[theorem]{Corollary}
\newtheorem{fact}[theorem]{Fact}
\newtheorem{lemma}[theorem]{Lemma}
\newtheorem{proposition}[theorem]{Proposition}
\theoremstyle{definition}

\newcommand{\prob}{\mathbb{P}}

\newcommand{\bG}{\ensuremath{\mathbf{G}}}
\newcommand{\rG}{\ensuremath{\mathrm{G}}}
\newcommand{\rH}{\ensuremath{\mathrm{H}}}
\newcommand{\bQ}{\ensuremath{\mathbf{Q}}}
\newcommand{\rQ}{\ensuremath{\mathrm{Q}}}
\newcommand{\bT}{\ensuremath{\mathbf{T}}}
\newcommand{\rT}{\ensuremath{\mathrm{T}}}

\newcommand{\cT}{\ensuremath{\mathcal{T}}}
\newcommand{\cQ}{\ensuremath{\mathcal{Q}}}
\newcommand{\cS}{\ensuremath{S}}

\title{Random infinite squarings of rectangles}
\author{Louigi Addario-Berry}
\address{Department of Mathematics and Statistics, McGill University, 805 Sherbrooke Street West, 
		Montr\'eal, Qu\'ebec, H3A 2K6, Canada}
\email{louigi.addario@mcgill.ca}
\urladdr{http://www.math.mcgill.ca/~louigi/}
\thanks{L.A.-B. was supported by an NSERC Discovery Grant and an FQRNT Nouveaux Chercheurs Grant during this research.} 

\author{Nicholas Leavitt}
\address{Department of Mathematics and Statistics, McGill University, 805 Sherbrooke Street West, 
		Montr\'eal, Qu\'ebec, H3A 2K6, Canada}
\email{nicholas.leavitt@mail.mcgill.ca}

\date{May 12, 2014}
\begin{document}
%\setmainfont{Baskerville}

\begin{abstract}
A recent preprint \cite{ab14hex} introduced a growth procedure for planar maps, whose almost sure limit is ``the uniform infinite $3$-connected planar map". 
A classical construction of Brooks, Smith, Stone and Tutte \cite{brooks40dissection} associates a squaring of a rectangle (i.e. a tiling of a rectangle by squares) to any to finite, edge-rooted planar map with non-separating root edge. We use this construction together with the map growth procedure to define a growing sequence of squarings of rectangles. We prove the sequence of squarings converges to an almost sure limit: a random infinite squaring of a finite rectangle. This provides a canonical planar embedding of the uniform infinite $3$-connected planar map. We also show that the limiting random squaring almost surely has a unique point of accumulation. 
\end{abstract}

\maketitle

\begin{figure}[ht]
\begin{subfigure}[b]{0.48\textwidth}
\includegraphics[width=\linewidth]{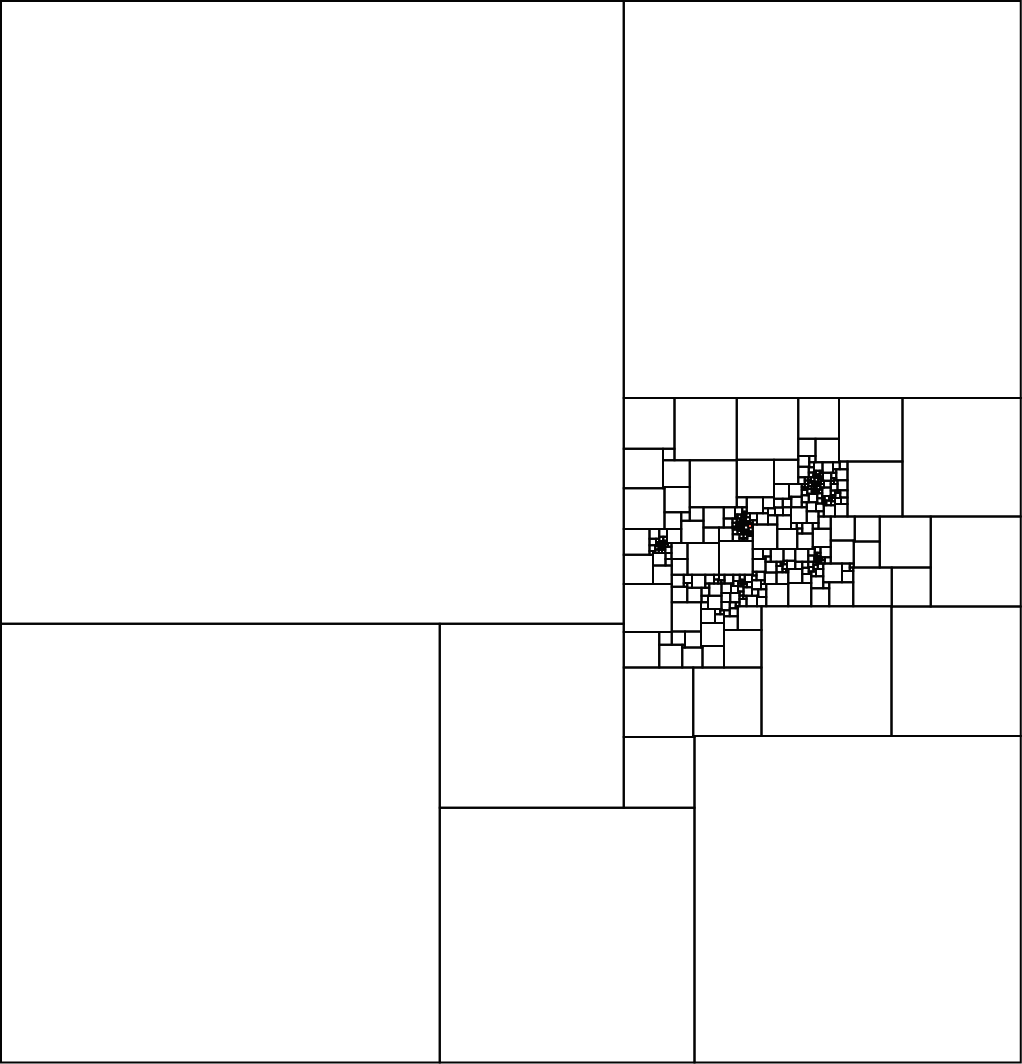}
\caption{\small A random squaring of a rectangle with 6345 squares. }
\end{subfigure}
\quad
\begin{subfigure}[b]{0.48\textwidth}
\includegraphics[width=\linewidth]{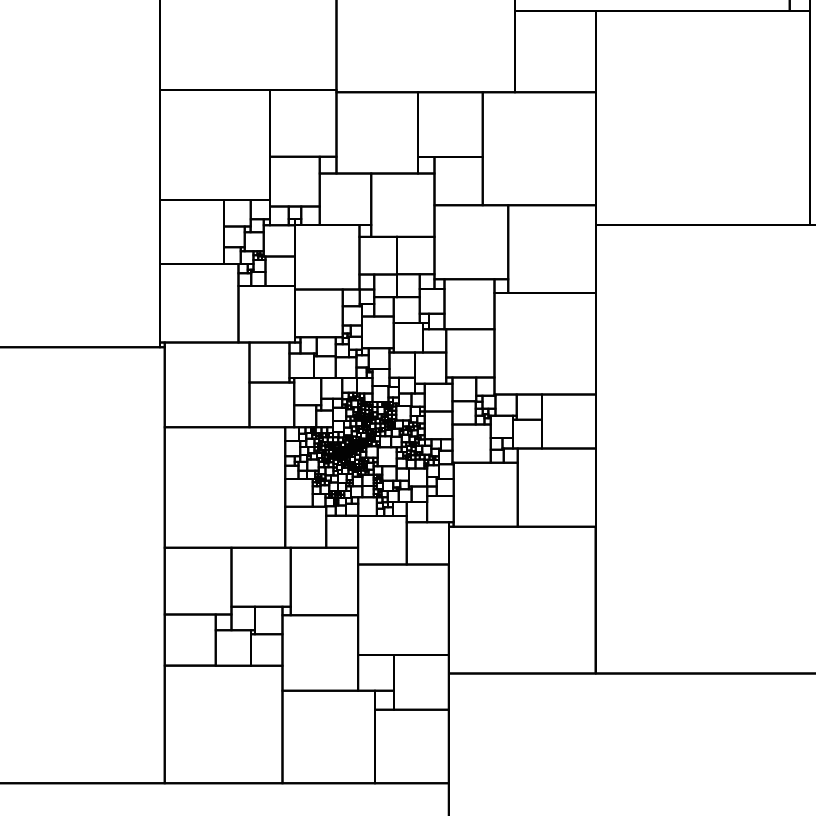}
\caption{\small A magnification of a small region within the same squaring.}
\end{subfigure}
%\label{fig:first}
%\caption{The region in (B) is contained within the red dot in (A), which may be seen in electronic versions of this file by zooming in.}
\end{figure}

\section{\large {\bf Introduction}}
In the proceedings of the 2006 ICM, Oded Schramm \cite{schramm2007icm} suggested the problem of determining the Gromov-Hausdorff (GH) distributional limit of uniform triangulations of the sphere and noted the connection, predicted in the physics literature, between such a limit and what he called ``the enigmatic KPZ formula \ldots relating exponents in quantum gravity to the corresponding exponents in plane geometry". Since that time, research in the area has exploded. Le Gall's \cite{legall2013brownianmap} proof that uniform triangulations have the metric space called the {\em Brownian map} as their GH limit\footnote{Independently and roughly simultaneously, Miermont \cite{miermont2013brownianmap} proved a similar result for uniform quadrangulations.}, and Duplantier and Sheffield's \cite{duplantier2011liouville} rigorous formulation and proof of the KPZ scaling relation for ``Liouville quantum gravity" (LQG) are two recent highlights within the subject. 

It has proved challenging to establish a direct connection between the Brownian map and LQG. \footnote{We should note that the very recent preprint of Miller and Sheffield \cite{miller14qle} announces a new project, whose ``ultimate aim ... is to rigorously construct the metric space structure of the corresponding [$\gamma^2=8/3$] LQG surface and to show that the random metric space obtained this way agrees in law with a form of the Brownian map."}
A major obstacle is that although the Brownian map is known to have the {\em topology} of the $2$-sphere $\mathbb{S}^2$ \cite{legall2007topological,miermont2008sphericity}, GH convergence alone is not strong enough to yield a canonical random metric (distance function) $d$ on $\mathbb{S}^2$ so that $(\mathbb{S}^2,d)$ has the law of the Brownian map. In seeking such a metric $d$, it is natural to consider discrete models of maps that at least have natural canonical embeddings in $\mathbb{S}^2$. Such considerations led Le Gall \cite{legall2013brownianmap,legall14rg} to suggest the study of uniform $3$-connected triangulations, for which one may use the associated circle packings to define canonical embeddings.\footnote{The circle-packing theorem states that such an embedding is unique up to conformal automorphisms.} 

In this paper, we instead study uniform $3$-connected planar maps (or equivalently, by Whitney's theorem, planar graphs) without constraints on face degrees.\footnote{A graph $G$ is $3$-connected if the removal of any two vertices $v,v'$ leaves $G$ connected. Similarly, a graph $G$ is $2$-connected if it has no {\em cut-vertex}, i.e., no vertex $v$ whose removal disconnects $G$.} A classic result of Brooks', Smith, Stone and Tutte \cite{brooks40dissection} canonically associates a {\em squaring of a rectangle} to any finite edge-rooted planar map $\mathrm{G}=(G,st)$ whose root edge is not a cut-edge; furthermore, any squaring of a rectangle may be so obtained. We write $\cS(\mathrm{G})$ for this squaring, which is the union of a collection of line segments in $\mathbb{R}^2$. Here is the main result of the current paper. 
\begin{theorem}\label{thm:main}
There exists an explicitly defined sequence $(S_n,n \ge 1)$ of random squarings of rectangles, $S_n$ being composed of $n$ squares, which converges almost surely for the Hausdorff distance to a compact limit $S_{\infty}$. Furthermore, $S_{\infty}$ a.s.\ has exactly one point of accumulation and has the law of $\cS(\mathbf{G}_{\infty})$, where $\mathbf{G}_{\infty}$ is ``the uniform infinite $3$-connected planar map".
\end{theorem}
For a given graph $\mathrm{G}=(G,st)$, the calculation of $\cS(\mathrm{G})$ is accomplished by viewing $\mathrm{G}$ as an electrical network with potential difference one between the ends of the root edge, finding the potentials at the vertices of $G$, then applying a simple geometric construction which we shortly describe. This is all accomplished by solving the equations given by Kirchoff's laws. Furthermore, the resulting geometric representation is closely linked to the properties of simple random walk on $G$. The simplicity and explicit nature of the construction, the connection with electrical networks, and the definition of $S_{\infty}$ as an almost sure (rather than distributional) limit together lead us to view Theorem~\ref{thm:main} as a promising tool for connecting random maps with LQG. Several precise questions, some in this vein, appear in Section~\ref{sec:quests}. 

We conclude the introduction with a brief sketch of what follows. Section~\ref{sec:preliminaries} primarily introduces the objects of study and describes existing results of which we make use. In particular, in Section~\ref{ssec:squaring} we describe the Brooks-Smith-Stone-Tutte construction of squarings from planar maps. 
In Section~\ref{ssec:bijec}, we construct the sequence $(S_n, 1 \le n \le \infty)=(\cS(\mathbf{G}_n),1 \le n \le \infty)$ from an a.s.\ convergent sequence $(\mathbf{G}_n,1 \le n \le \infty)$ of random maps introduced in a recent preprint of the first author \cite{ab14hex}; we then establish the Hausdorff convergence of $S_n$ to $S_{\infty}$ in Section~\ref{sec:convergence}. In Section~\ref{sec:onepoint} we analyze the contacts graph of the limit $S_{\infty}$, showing that it is a.s.\ vertex-parabolic and one-ended. Theorem~\ref{thm:main} is then easily deduced in Section~\ref{sec:proof}. Finally, Section~\ref{sec:quests} contains questions and conjectures.

\section{\large {\bf Preliminaries: graph limits, squarings, bijections, and recurrence.}} \label{sec:preliminaries}

\subsection{Terminology} 
For the remainder of the paper, all graphs are assumed to be simple and have finite degrees unless otherwise indicated. For any graph $G$, write $v(G)$ and $e(G)$ for the vertices and edges of $G$, respectively, so $G=(v(G),e(G))$. We write $\deg_G(v)$ for the degree of $v \in v(G)$. For $r > 0$, write $B_G(v,r)$ for the subgraph of $G$ induced by vertices at graph distance at most $r$ from $v$. Given $A \subset e(G)$ we write $G-A$ for the graph $(v(G),e(G)\setminus A)$, and given $U \subset v(G)$ we write $G-U$ for the subgraph of $G$ induced by $v(G)\setminus U$. Finally, a {\em rooted graph} is a pair $(G,v)$ where $G$ is a graph and $v \in v(G)$. 

\subsection{Distributional limits of graphs} \label{ssec:distlim}%
Given rooted graphs $\rG=(G,\rho)$ and $\rG'=(G',\rho')$, we say the {\em distance} between $\rG$ and $\rG'$ is $1/(r+1)$, where $r$ is the greatest value for which $(B_G(\rho,r),\rho)$ and $(B_{G'}(\rho',r),\rho')$ are isomorphic (as rooted graphs). The distance between two graphs is zero precisely if they are isomorphic, so it is straightforward to show this distance defines a metric on the set of isomorphism classes of locally finite rooted graphs.  Convergence for this metric is often called {\em local weak} or {\em Benjamini-Schramm} convergence of graphs \cite{benjamini01recurrence,aldous04objective}. A sequence $((G_n,\rho_n),n \ge 1)$ converges in the local weak sense precisely if there exists a graph $(G_\infty,\rho_{\infty})$ such that for any $r > 0$ and all $n$ sufficiently large, $(B_{G_n}(\rho_n,r),\rho_n)$ and $(B_{G_{\infty}}(\rho_{\infty},r),\rho_{\infty})$ are isomorphic. A sequence $((G_n,\rho_n),n \ge 1)$ of {\em random} rooted graphs converges in {\em distribution} in the local weak sense if for every finite rooted graph $(G,\rho)$, $\prob(B_{G_n}(\rho_n,r)=(G,\rho))$ converges as $n \to \infty$. 

Recall that simple random walk on a locally finite graph $G$ is the Markov chain $(X_n,n \ge 0)$ on $v(G)$ with transition probabilities $p_{uv}=1/\deg(u)$ for $\{u,v\} \in e(G)$, and $p_{uv}=0$ otherwise. If $G$ is finite, then the simple random walk has stationary measure $\pi$ given by $\pi(v) = \deg_G(v)/2|e(G)|$ for all $v \in v(G)$. 
The graph $G$ is \emph{recurrent} if for all $v \in v(G)$, $\prob(\exists n>0:X_n = v|X_0=v) = 1$. Equivalently, $G$ is recurrent if and only if when edges are viewed as unit resistors, the electrical resistance from any node to infinity is infinite. 

We next state a beautiful theorem of Gurel-Gurevich and Nachmias \cite{gg13recurrence}, which we use below. 
A random infinite graph $(G,\rho)$ is a {\em distributional limit of finite planar graphs} if there exists a sequence $(G_n,\rho_n)$ of random finite planar graphs such that (a) for each $n$, conditional on $G_n$, the root $\rho_n$ is distributed according to the stationary measure on $v(G_n)$,\footnote{See \cite{benjamini01recurrence} for a more detailed discussion of this condition.} and (b) $(G_n,\rho_n)$ converges in distribution in the local weak sense to $(G,\rho)$. Finally, say a random variable $X$ has {\em exponential tail} if there exists $c > 0$ such that $\prob(X>t) \le e^{-c t}$ for all sufficiently large $t$. 
\begin{theorem}[\cite{gg13recurrence}, Theorem 1.1]\label{thm:limrecur}
Let $(U,\rho)$ be a distributional limit of finite planar graphs such that the degree of $\rho$ has exponential tail. Then $U$ is almost surely recurrent.
\end{theorem}

\subsection{Squarings of rectangles} \label{ssec:squaring}
An \emph{edge-rooted map} is a pair $(G,{st})$, where $G$ is a connected planar graph, properly embedded in $\mathbb{R}^2$ such that the distinguished directed edge $st$ lies in the unique unbounded face of $G-\{s,t\}$. We write $(G^*,s^*t^*)$ for the planar dual of $(G,{st})$, with the convention that the tail $s^*$ of $s^*t^*$ is in the face lying to the right of ${st}$. For $e \in e(G)$ write $e^*$ for the dual edge to $e$ in $(G^*,s^*t^*)$. A map is {\em locally finite} if all vertices and faces have bounded degree. Throughout this section, $\mathrm{G}=(G,{st})$ denotes a fixed, locally finite edge-rooted map such that $G-\{s,t\}$ is connected (in this case $G^*-\{s^*,t^*\}$ is also connected).

A \emph{squaring of a rectangle} is a closed set $S \subset \mathbb{R}^2$ such that all bounded components of $\mathbb{R}^2\setminus S$ are (open) squares, and the closure of the union of all such bounded components is a compact rectangle. The {\em squares} of $S$ are the closures of the connected components of $\mathbb{R}^2\setminus S$ (note that here we include the unbounded component). It is easily seen that $S$ is recoverable from its set of squares.

We now define the squaring $\cS(\mathrm{G})$ associated to an edge-rooted map $\rG=(G,{st})$; this construction was discovered by Brooks, Smith, Stone, and Tutte \cite{brooks40dissection}. (Another, rather different way to define squarings using maps was later described by Schramm \cite{schramm93square}.) The definition is illustrated in Figure~\ref{fig:squaringwithcoords}. 
First associate an electrical network with $\rG$ as follows. Cut ${st}$, connect a 1 volt power supply to $s$, ground at $t$, and let edges act as unit resistors. Write $\lambda(\rG)$ for the total current flowing from $s$ to $t$. For $v \in v(G)$, write $P(v)$ for the potential at $v$ (equivalently, $P(v)$ is the probability a simple random walk starting from $v$ first visits $s$ before first visiting $t$). We note the following identity, which is an immediate consequence of conservation of current flow and Ohm's law, for later reference: 
\begin{equation}\label{eq:cursum}
\lambda(\rG) = \sum_{v \sim s,v\ne t} ( 1 - P(v)) = \deg_G(s) - 1 - \sum_{v \sim s} P(v)\, .
\end{equation}

Next, to each edge $e \in e(G)\setminus\{s,t\}$, associate a square $s_e$ whose side length is equal to the current flowing through $e$. 
The position of $s_e$ is determined as follows. With $e=\{u,v\}$, let $y(e)=\max(P(u),P(v))$. 
%Next, view $G^*-\{s^*,t^*\}$ as an electrical network with potential $\lambda(\rG)$ at $s^*$ and grounded at $t^*$, where $\lambda(\rG)$ is the total current flowing through $\rG$. 
Finally, view $(G^*, s^*t^*)$ as an electrical network with potential $\lambda(\rG)$ at $s^*$ and grounded at $t^*$. Then 
for $e \in e(G)\setminus\{s,t\}$, with $e^*=\{u^*,v^*\}$, let $x(e^*)=\min(P(u^*),P(v^*))$. 
The top left corner of $s_e$ then has position $(x(e^*),y(e))$. 

Let $\cS(G,st)$ be the union of the boundaries of the squares $\{s_e, e \in e(G)\setminus \{s,t\}\}$. It is straightforward to show that $\cS(G^*,s^*t^*)$ may be obtained by rotating $\cS(G,st)$ counterclockwise by $90^{\circ}$, then translating and rescaling so that the squaring has height one and bottom left corner at the origin; this will be useful later. 
\begin{theorem}[\cite{brooks40dissection,benjamini96rw}]\label{thm:bsst}
Let $\rG=(G,st)$ be a finite, edge rooted, finite planar map such that $G-\{s,t\}$ is connected. Then the squares $\{s_e,e \in e(G)\setminus \{s,t\}\}$ have disjoint interiors, and $\cS(\rG)$ is a squaring of the rectangle $[0,\lambda(\rG)]\times[0,1]$.  Furthermore, $\cS$ is invertible up to zero current edges.
\end{theorem}
Figure~\ref{fig:squaringwithcoords} contains an edge-rooted map $(G,st)$ with ten edges, and its associated squaring.

Theorem~\ref{thm:bsst} applies to finite graphs, but $\cS(\rG)$ is defined whenever $\rG$ and $\rG^*$ are locally finite (and $G-\{s,t\}$ is connected).
  In this case, there is no guarantee that $\cS(\rG)$ is a squaring of a rectangle, but this is known to hold in some cases (see, e.g., \cite{benjamini96rw}). Proposition~\ref{prop:conv_squaring}, below, implies that $\cS(\rG)$ is a squaring whenever $\rG$ is recurrent, a fact which we require and were unable to find in the literature. 

\begin{figure}[ht]
\includegraphics[width=.85\linewidth,page=2]{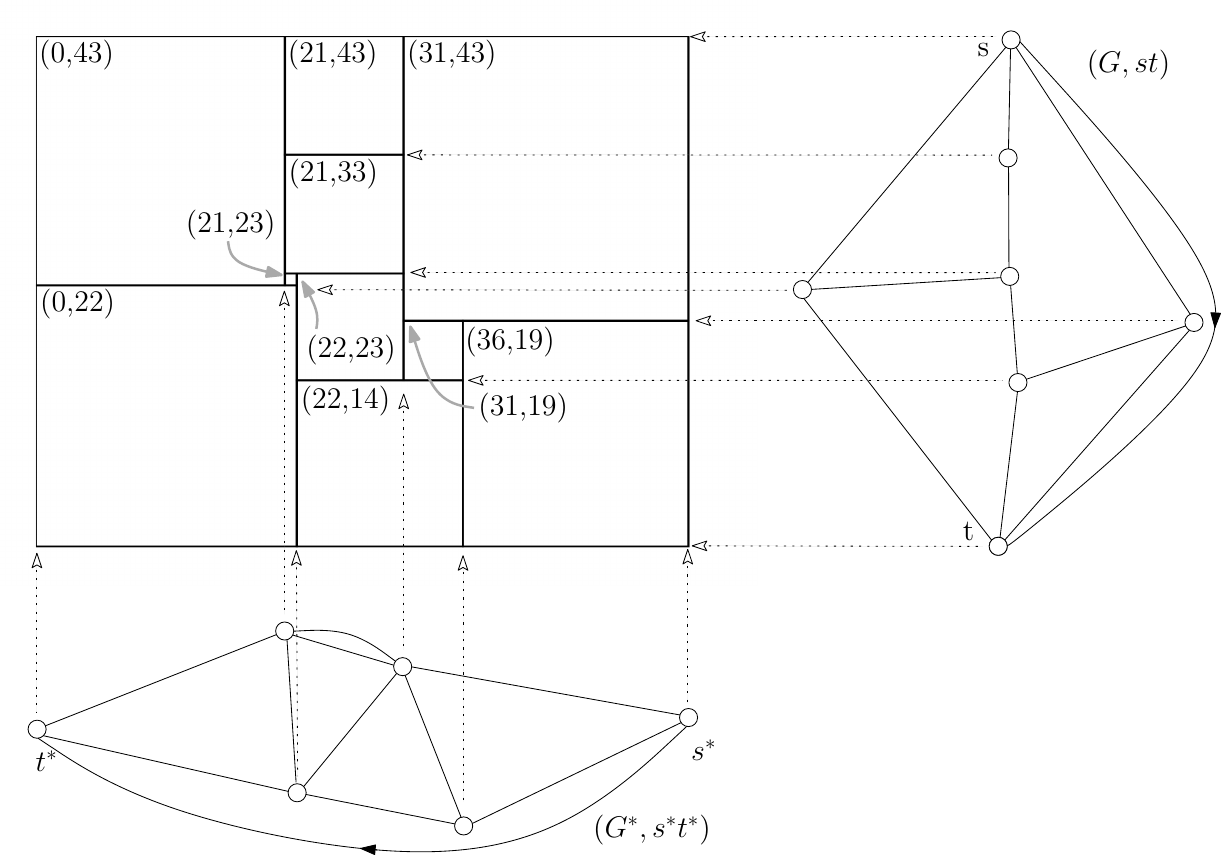}
\caption{\Small A squaring $S$ with ten squares, and the associated edge-rooted map $G$. The coordinates of the squares' top left corners, rescaled to be integers, are displayed. Horizontal lines connect vertices $v$ of $G$ with the associated line segments $\ell(v) \subset S$; vertical lines connect faces (dual vertices) with the associated line segments.}
\label{fig:squaringwithcoords}
\end{figure}

For the remainder of the section, we assume $(G,st)$ is such that the squares $\{s_e,e \in e(G)\setminus \{s,t\}\}$ have disjoint interiors and such that $\cS(G,st)$ is a squaring of a rectangle (by Theorem~\ref{thm:bsst}, this is the case if $G$ is finite, but we do not assume finiteness). 
For $v \in v(G)$, define a horizontal line segment $\ell(v)$ contained in $\cS(G,st)$ as follows (see Figure~\ref{fig:squaringwithcoords}). 
The $y$-coordinate of $\ell(v)$ is $P(v)$. Next, the leftmost (resp.\ rightmost) point of $\ell(v)$ is the minimal (resp.\ maximal) $x$-coordinate contained in any square $s_e$ for which $e$ is incident to $v$. (If no current flows through $v$ then $\ell(v)$ consists of a single point; in this case we call $\ell(v)$ {\em degenerate}, and otherwise we call $\ell(v)$ {\em non-degenerate}.) Likewise, to each face $f$ of $G$ we associate a vertical line segment $\ell(f)$, which may either be defined directly or using the above observation that the squaring of a graph and of its facial dual are related by a rotation; we omit details. We call $\ell(v)$ and $\ell(f)$ {\em primal} and {\em facial} lines of $\cS(G,st)$, respectively. 

Observe that if $e$ is incident to vertex $v$ (resp.\ face $f$) then $s_e \cap \ell(v)$ is a horizontal border of $s_e$ (resp.\ $s_e \cap \ell(f)$ is a vertical border of $s_e$).  Similarly, if vertex $v$ is incident to face $f$ in $G$ then there is an edge $e$ incident to both $v$ and $f$, from which it follows that $\ell(v) \cap \ell(f) \ne \emptyset$. The disjointness of the interiors of the squares $\{s_e,e \in e(G)\setminus \{s,t\}\}$ implies that any distinct lines $\ell,\ell'$, whether primal or facial, are either disjoint or intersect in a single point. It follows that if $\{u,v\} \in e(G)$ then the top and bottom borders of $s_e$ are both contained in $\ell(v) \cup \ell(w)$. 

Given a squaring of a rectangle $S$ (with lower left corner at the origin), one may define an edge-rooted map $(G',s't')$ with $\cS(G',s't')=S$ as follows. The vertices of $G'$ are the maximal horizontal line segments of $S$. For each square $s$ of $S$, there is an edge connecting the vertices $\ell,\ell'$ of $G'$ that border the top and bottom of $s$, respectively. The graph $(G',s't')$ need not be even locally finite. However, when $(G',s't')$ is finite then $\cS(G',s't')=S$ (see \cite{brooks40dissection}, Theorem 4.31). 

Despite the construction of the preceding paragraph, the function $\cS(\cdot)$ is not invertible: if the current through $e \in e(G)\setminus \{s,t\}$ is zero, and $\hat{G}$ is obtained from $G$ by either deleting or contracting $e$, then then $\cS(\hat{G},st) = \cS(G,st)$; see Figure \ref{fig:badsquaring}. However, it is not hard to see that zero-current edges are the only way injectivity can fail. In particular, if $S$ is such that no four squares have a common point of intersection, then there is a unique 2-connected edge-rooted map $(G',s't')$ with squaring $S$. 

We conclude the section with a lemma. 
\begin{lemma}\label{lem:3con_nondegen}
If $(G,st)$ is a finite $3$-connected edge-rooted planar map, then for all vertices $v$ (resp.\ faces $f$) of $G$, $\ell(v)$ (resp.\ $\ell(f)$) is non-degenerate.
\end{lemma}
\begin{proof}
Let $(G,st)$ be $3$-connected. 
Suppose there exists $w \in v(G)$ such that $\ell(w)$ is a single point, say $\ell(w)=z \in \mathbb{R}^2$. Then any edge $e$ incident to $w$ has $s_e = z$. Since for any neighbour $v$ of $w$, $\ell(w)$ is a border of $s_{\{v,w\}}$, for such $v$ we have $z \in \ell(v)$. Letting $U = \{u \in v(G): z \in \ell(u),u\mbox{ is non-degenerate}\}$, it follows that $U$ is a cutset in $G$ separating $w$ (and any other vertices $w'$ with $\ell(w') =z$) from $s$ and $t$. But since distinct lines are either disjoint or intersect in a single point, $U$ has size at most two, which contradicts that $G$ is $3$-connected.
\end{proof}

\begin{figure}[ht]
\includegraphics[width=0.85\linewidth]{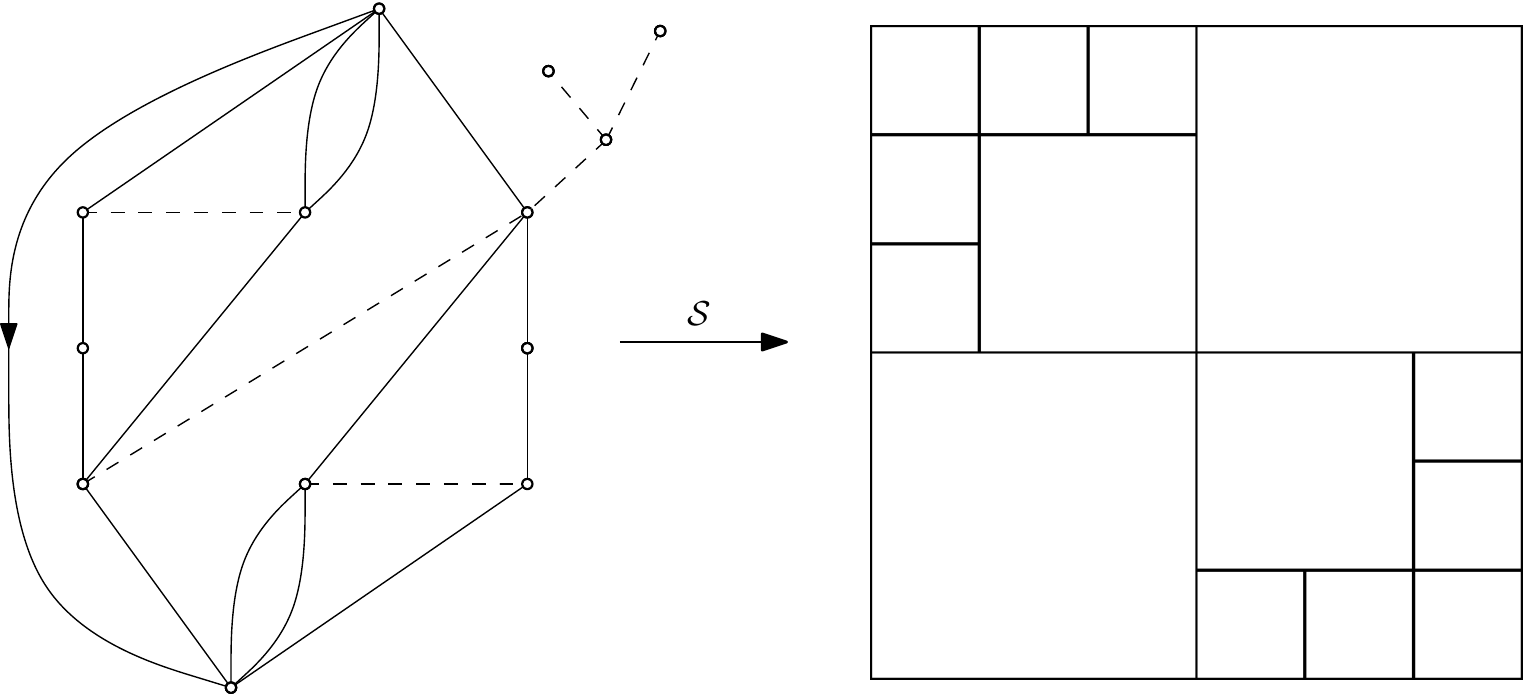}
\caption{\small An example of the non-invertibility of $\cS$. Removing or contracting any of the dotted edges leaves the squaring unchanged.} 
\label{fig:badsquaring}
\end{figure}

\subsection{Bijections for random trees and maps} \label{ssec:bijec}

Let $\rT=(T,uv)$ be a finite edge-rooted map which is a {\em binary tree}: that is, each non-leaf node of $T$ has degree three. Fusy, Poulalhon and Schaeffer \cite{fusy08dissection} described an invertible ``closure" operation which transforms $T$ into an edge-rooted {\em irreducible quadrangulation of a hexagon}, which we denote $\rQ(\rT)$.\footnote{An edge-rooted map is an irreducible quadrangulation of a hexagon if the unbounded face has degree six, all other faces have degree four, and every cycle of length four bounds a face.} The construction of $\rQ(\rT)$ from $T$, which we now explain, is illustrated in Figure~\ref{fig:closure}. Perform a clockwise contour exploration of $T$. Each time a leaf $w$ is followed by four internal vertices $x,y,z,a$, identify $w$ with $a$ so that the unbounded face lies to the left of the oriented edge $wa$; the face to the right will necessarily have degree four. 
\begin{figure}[ht]
\begin{subfigure}[b]{0.3\textwidth}
		\includegraphics[width=\textwidth,page=1]{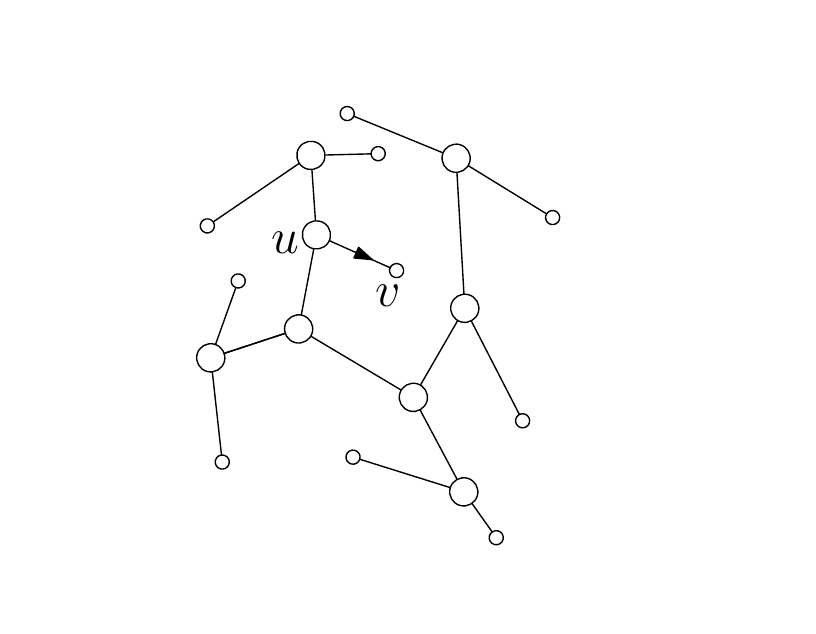}
                \caption{A binary tree $\rT$.\\}
                \label{fig:closure-inner}
\end{subfigure}%
\quad
\begin{subfigure}[b]{0.3\textwidth}
		\includegraphics[width=\textwidth,page=2]{figures/treetohex.pdf}
                \caption{The partial closure $Q_0$.\\ }
                \label{fig:closure-hex}
\end{subfigure}%
\quad
\begin{subfigure}[b]{0.3\textwidth}
		\includegraphics[width=\textwidth,page=4]{figures/treetohex.pdf}
                \caption{The map $\rQ(\rT)$.}
                \label{fig:closure-orient}
\end{subfigure}%
\quad
\caption{Closure edges are drawn with two colours: blue and black if internal, red and black if incident to the hexagonal face.} 
\label{fig:closure}
\end{figure}
In the modified map, the vertex formed from merging $w$ and $a$ is considered to be internal. Continue exploring the modified map in clockwise fashion, making identifications according to the preceding rule, until no identifications are possible; call the result the {\em partial closure} of $T$, and denote it $Q_0$ (see Figure~\ref{fig:closure-hex}). A counting argument shows that in $Q_0$, at least $3$ leaves remain. Write $\ell_0$ for the first such leaf encountered by the contour process (in clockwise order starting from $uv$). 

Draw a hexagon in the unbounded face of $Q_0$. It can be shown that there is a unique (up to isomorphism of planar maps) way to identify the leaves of $Q_0$ with vertices of the hexagon so that in the resulting graph, all bounded faces have degree four. The result is the graph $\rQ(\rT)$, which we view as rooted at $uv$ (the vertex $v$ may have been identified with another vertex during the closure operation, as in Figure~\ref{fig:closure-orient}; we abuse notation and continue to use the same name). 
Write $\cT_n$ for the set of edge-rooted binary trees with $n$ internal vertices, and $\cQ_n$ for the set of edge-rooted quadrangulations of a hexagon with $n+6$ vertices, such that the root edge is not incident to the hexagonal face. 

\begin{theorem}[\cite{fusy08dissection}]
For each $n \ge 1$, the closure operation is a bijection between $\cT_n$ and $\cQ_n$. 
\end{theorem}
Given $\rQ \in \cQ_n$, number the vertices of the hexagonal face in clockwise order as $0,\ldots,5$, where $0$ is the vertex identified with $\ell_0$ by the closure operation. Then, for $i \in \{0,\ldots,5\}$, let $\rQ^{(i)}$ be obtained from $\rQ$ by adding the oriented edge from $i$ to $(i+3)\mod 6$. The result is a doubly edge-rooted quadrangulation (every face has degree four) which may no longer be irreducible. However, we do have the following. Let $\cQ_n^*$ be the set of triples $(Q,uv,u'v')$, where $Q$ is an irreducible quadrangulation with $n+6$ vertices and $uv,u'v'$ are oriented edges of $Q$ not lying on a common face, and let $\cT_n^* = \{(\rT,i) \in \cT_n \times \{0,\ldots,5\}: \rQ(\rT)^{(i)} \in \cQ_n^*\}$. 
\begin{theorem}[Fusy Thm 4.8]
The function sending $(\rT,i)$ to $\rQ(\rT)^{(i)}$ restricts to a bijection between $\cT_n^*$ and $\cQ_n^*$. 
\end{theorem}
To conclude the section, we recall Tutte's classical bijection between maps and quadrangulations, which associates to a rooted map $(G,uv)$ a quadrangulation $(Q,uv')$ as follows (see Figure~\ref{fig:trivialbij} for an illustration). 
 \begin{wrapfigure}[16]{R}{.4\textwidth}
\vspace{-0.3cm}
\rule[5pt]{.4\textwidth}{0.5pt}
		\includegraphics[width=0.4\textwidth,page=6]{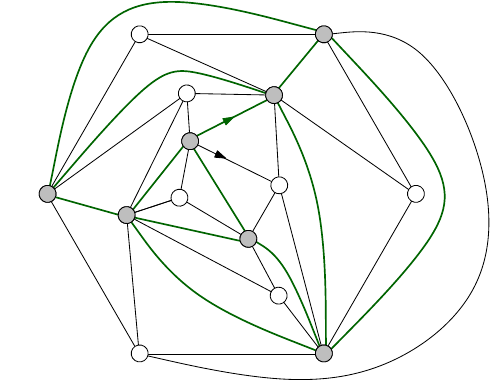}
		%\caption{\small In green, an edge-rooted map; in black, its image under Tutte's bijection.}
\caption{\small An edge-rooted map (grey vertices; thicker, green edges), and its image under Tutte's bijection (grey and white vertices; thinner, black edges).}
\label{fig:trivialbij}
\rule[8pt]{.4\textwidth}{0.5pt} 
\end{wrapfigure}
Draw a node $v_f$ in each face $f$ of $G$. Then for each $u \in v(G)$ and each face $f$ incident to $u$, add an edge $\{u,v_f\}$. Erase all edges in $E(G)$. Finally, let $v'=v_f$ where $f$ is the face lying to the right of $uv$. It is well-known (see, e.g., \cite{fusy08dissection}, Theorem 3.1) that Tutte's bijection restricts to a bijection between $3$-connected, edge-rooted planar maps with $n$ edges and edge-rooted, irreducible quadrangulations with $n$ faces. Thus, combining the closure operation with Tutte's bijection yields a bijection between $\cT_n$ and a collection of edge-rooted maps which includes all $3$-connected maps with $n+4$ edges. It turns out that the $3$-connected maps comprise an asymptotically constant proportion of the collection; we return to this point in the next section.

\subsection{Growth procedures}\label{ssec:growth}
In this section, we describe the growth procedures, introduced in \cite{ab14hex}, for random irreducible quadrangulations of a hexagon and to the maps associated to such quadrangulations by Tutte's bijection. We restrict our discussion to the features of the procedures required in the current work, and refer the reader to \cite{ab14hex} for further details. 

Luczak and Winkler \cite{luczak04building} showed that there is a growth procedure for uniformly random binary plane trees. More precisely, there exists a stochastic process $(\bT_n)_{n=0}^\infty$, with the following properties. First, for each $n$, $\bT_n$ is uniformly distributed over edge-rooted binary trees with $n$ internal nodes. Second, for all $n$, $\bT_n$ is a subtree of $\bT_{n+1}$. 

The sequence $(\bT_n,n \ge 1)$ converges almost surely in the local weak sense to a limit $\bT_{\infty}$, which is essentially a critical Galton-Watson tree whose offspring law $\mu$ satisfies $\mu(\{0\})=\mu(\{2\})=1/2$, conditioned to be infinite. It is shown in \cite{ab14hex} that the closure operation, when applied to $\bT_{\infty}$, yields an infinite, locally finite edge-rooted quadrangulation $\bQ_{\infty}$, and that $\rQ(\bT_n) \to \bQ_{\infty}$ almost surely, in the local weak sense.\footnote{In the partial closure of $\bT_{\infty}$, almost surely no leaves remain, so the addition of the ``external hexagon" does not occur. Thus, $\bQ_{\infty}$ is perhaps more accurately described as the partial closure of $\bT_{\infty}$.}

Now let $i$ be uniform in $\{0,\ldots,5\}$, let $(Q_n,st')$ be the (singly) edge-rooted quadrangulation obtained from $\rQ(\bT_n)^{(i)}$ by unrooting at the second root edge (but not deleting the edge), and let $\bG_n=(G_n,st)$ be the pre-image of $(Q_n,st')$ under Tutte's bijection. Then $(Q_n,st')$ is irreducible if and only if $\bG_n$ is $3$-connected. Fusy, Poulalhon and Schaeffer \cite{fusy08dissection} show that $\prob(\bG_n\mbox{ is $3$-connected}) \to 2^8/3^6$ as $n \to \infty$.\footnote{Note that this is a statement about large-$n$ asymptotics of marginal probabilities, and says nothing about the {\em dynamics} of $(\bG_n,n \ge 1)$.}  They further deduce from the bijective results described in Section~\ref{ssec:bijec} that the conditional law of $\bG_n$, given that $\bG_n$ is $3$-connected, is uniform over $3$-connected, edge-rooted maps with $n+4$ edges. 

It is an easy consequence of the convergence of $\rQ(\bT_n)$ to $\bQ_{\infty}$ that $\bG_n$ has an a.s. local weak limit $\bG_{\infty}=(G_{\infty},st)$, and $\bG_{\infty}$ is obtained from $\bQ_{\infty}$ via Tutte's bijection. We also have the following result.
\begin{theorem}[\cite{ab14hex}, Theorem 7]\label{thm:distlim}
Let $\hat{\bG}_n$ be uniformly distributed on the set of 3-connected rooted maps with $n+4$ edges; then $\hat{\bG}_n$ converges in distribution to $\mathbf{G}_\infty$ in the local weak sense. 
\end{theorem}
In brief, Theorem~\ref{thm:distlim} holds because the failure of $\bG_{n}$ to be $3$-connected is caused by the failure of $\rQ(\bT_n)$ to be irreducible. This is a ``local defect", occurring near the hexagon, and the hexagon disappears to infinity in the limit. 

Theorem~\ref{thm:distlim} implies that $G_{\infty}$ is a.s.\ $3$-connected, so up to reflection has a unique planar embedding. (The definition of $G_{\infty}$ as an almost sure limit of a sequence of finite maps also {\em uniquely} specifies an embedding of $G_{\infty}$.) We thus henceforth view $G_{\infty}$ as a planar map. We conclude the section with a quick application of Theorem~\ref{thm:limrecur}. 
\begin{theorem}\label{thm:ginfinity_recur}
$G_{\infty}$ and its planar dual $G_{\infty}^*$ are both a.s.\ recurrent. 
\end{theorem}
In proving Theorem~\ref{thm:ginfinity_recur}, we use the following fact. 
\begin{fact}\label{fact:expvert}
The random variable $\deg_{\bG_{\infty}}(s)$ has exponential tail. 
\end{fact}
\begin{proof}
Let $\hat{\bG}_n=(\hat{G}_n,\hat{s}_n\hat{t}_n)$ be as in Theorem~\ref{thm:distlim}. Then for $d \in \mathbb{N}$, $\prob(\deg_{\bG_{\infty}}(s) =d) = \lim_{n \to \infty} \prob(\deg_{\hat{\bG}_n}(\hat{s}_n) =d)$. 
By \citep[Theorem 2.1 (a)]{bender89face}, for all $\epsilon > 0$ there exists $B>0$ such that for all $n \in \mathbb{N}$, 
\[ 
\prob(\deg_{\hat{\bG}_n}(\hat{s}_n) =d) < B\cdot\left ( \frac{1}{2} + \epsilon \right)^d.
\]
The fact follows. 
\end{proof}
\begin{proof}[Proof of Theorem~\ref{thm:ginfinity_recur}]
By Theorem~\ref{thm:distlim}, $(G_{\infty},s)$ is a distributional limit of finite planar graphs; by Fact~\ref{fact:expvert} its root degree has exponential tail. The a.s.\ recurrence of $G_{\infty}$ then follows from Theorem~\ref{thm:limrecur}. Next, 
let $\hat{\bG}_n$ be as in Theorem~\ref{thm:distlim}. Then $\hat{\bG}_n$ has the same law as its planar dual $\hat{\bG}^*_n$, and the latter converges in distribution to $\bG^*_{\infty}$, so $G_{\infty}^*$ is likewise a.s.\ recurrent. 
\end{proof}

Now for $1 \le n \le \infty$, let $S_n=\cS(\bG_n)$ be the squaring associated to $\bG_n$. In the following section, we prove the first part of Theorem~\ref{thm:main} by showing that $S_n \to S_{\infty}$ almost surely, for the Hausdoff distance, as $n \to \infty$. 

\section{\large {\bf Convergence of the squarings} }\label{sec:convergence}

Given a graph $G$, recall that a function $\varphi:v(G) \to \mathbb{R}$ is called \emph{harmonic} with boundary $D \subset v(G)$ if for all $v \in v(G)\backslash D$,
\[
\varphi(v) = \frac{1}{\deg(v)}\sum_{w \sim v} \varphi(w). 
\]
Let $\mathrm{X}=(X_n, n\geq 0)$ be a simple random walk on $G$. For any set $A \subset v(G)$, let 
\[
\tau_A =\tau_A(\mathrm{X}) = \inf \{n \in \mathbb{N}: X_n \in A \}
\]
be the first hitting time of $A$ by the walk. We recall the following standard theorem relating harmonic functions and simple random walks; see, e.g., \citep[Section 4.2]{norris98markov}. 
\begin{theorem}\label{thm:harmrw}
Let $G$ be a recurrent graph and $\varphi:v(G) \to \mathbb{R}$ be harmonic with finite boundary $D$. 
Then for all $v \in v(G)\setminus D$, $\varphi(v)=\mathbb{E}(\varphi(X_{\tau_D})~|~X_0=v)$.
\end{theorem}
In order to prove convergence of the squarings $S_n$ to $S_{\infty}$, we naturally require the potential at each vertex of $G_n$ to converge to its limiting value in $G_{\infty}$. We prove a slightly more general theorem from which such a convergence will follow. 
\begin{theorem}\label{thm:harmconv}
Let $(\rH_n,1 \le n \le \infty)=((H_n,\rho_n),1 \le n \le \infty)$ be a sequence of rooted graphs such that $\rH_n \to \rH_{\infty}$ in the local weak sense, and such that $H_{\infty}$ is recurrent. Fix $D \subset v(H_{\infty})$ finite, and for each $n \le \infty$ large enough that $D \subset v(H_n)$, let $\varphi_n: v(H_n) \to \mathbb{R}$ be a harmonic function on $H_n$ with boundary $D$. Suppose further that the functions $\varphi_n$ all agree on $D$. Then for all $v \in v(G_{\infty})$, $\varphi_n(v) \to \varphi_{\infty}(v)$ as $n \to \infty$. 
\end{theorem}
\begin{proof}
Fix $v \in v(H_{\infty})$, and let $n_0$ be large enough that $D\cup \{v\} \subset v(H_n)$ for all for $n \ge n_0$. For $n_0 \le n \le \infty$ write $\mathrm{X}^{(n)}=(X_k^{(n)},k \ge 0)$ for simple random walk on $H_n$ started from $X^{(n)}_0=v$. In view of Theorem \ref{thm:harmrw} and finiteness of $D$, it is enough to show that for each $b \in D$, $\prob(X^{(n)}_{\tau_D(\mathrm{X}^{(n)})} = b) \to \prob(X^{(\infty)}_{\tau_D(\mathrm{X}^{(\infty)})} = b)$ as $n \to \infty$. In what follows we write, e.g., $\tau_D$ instead of $\tau_D(\mathrm{X}^{(n)})$ for readability; the omitted argument should be clear from context.

Let $E_n(r) = \{\tau_{B_{H_n}(v,r)^c} < \tau_D\}$ be the event that $\mathrm{X}^{(n)}$ reaches distance $r$ from $v$ before hitting $D$.  Since $H_\infty$ is recurrent, 
$\prob(E_{\infty}(r)) \to 0$ as $r \to \infty$. 

Next, fix $r_0$ large enough that $D \subset B_{H_\infty}(v,r_0)$. 
For $r\ge r_0$, by taking $N=N(r)$ large enough that $B_{H_n}(v,r+1)$ is constant for $n > N$, for such $n$ and for all $u \in D$ we have 
\[
\prob(E_{\infty}(r),X^{(n)}_{\tau_D} = u)=\prob(E_{n}(r),X^{(\infty)}_{\tau_D} = u).
\]
Summing over $u \in D$, this also implies that for such $n$, $\prob(E_{n}(r))=\prob(E_{\infty}(r))$. 

Now let $\epsilon > 0$ be arbitrary, fix $R> r_0$ large enough that $\prob_\infty(E(R)) < \epsilon$, and let $N=N(R)$ be as above. Then for $u \in D$, for $n > N$,
\begin{align*}
\prob(X^{(n)}_{\tau_D} = u)
&= \prob(E_{\infty}(r)^c,X^{(n)}_{\tau_D} = u) + \prob(E_{\infty}(v,r),X^{(n)}_{\tau_D} = u)\\
&<\prob(E_{\infty}(r)^c,X^{(n)}_{\tau_D} = u) + \epsilon\\
&= \prob(E_{n}(r)^c,X^{(\infty)}_{\tau_D} = u) + \epsilon\\
& \leq \prob(X^{(\infty)}_{\tau_D} = u) + \epsilon
\end{align*}
A symmetric argument shows that 
$\prob(X^{(n)}_{\tau_D} = u) > \prob(X^{(\infty)}_{\tau_D} = u) - \epsilon$, and thus 
\[
|\prob_n(X^{(n)}_{\tau_n} = \varphi(u)) - \prob_\infty(X_{\tau} = \varphi(u))|<\epsilon. 
\]
\end{proof}
\begin{proposition}\label{prop:conv_squaring}
Let $(\rH_n,1 \le n \le \infty)=((H_n,st),1 \le n \le \infty)$ be a sequence of locally finite, edge-rooted recurrent planar maps such that $H_{n}-\{s,t\}$ is connected for all $1 \le n \le \infty$ and such that $\rH_n \to H_{\infty}$ in the local weak sense. Then $\cS(\rH_{\infty})$ is a squaring of a rectangle, and $\cS(\rH_n) \to \cS(\rH_{\infty})$ as $n \to \infty$, for the Hausdorff distance. 
\end{proposition}
\begin{proof}
First assume that $H_n$ is finite for $1 \le n < \infty$. 
Let $\lambda_n$ be the total current flowing through $\rH_n$, and let $y_n : v(H_n) \to \mathbb{R}$ be the potential on $\rH_n$. Then let $x_n: v(H^*_n) \to \mathbb{R}$ be the potential on the dual graph $\rH^*_n$ when a potential of $\lambda_n$ is applied at $t^*$ and the graph is grounded at $s^*$. Now fix an edge $e=\{u,v\}$ of $H_{\infty}$. Then for $n$ sufficiently large that $e \in e(H_n)$, the square corresponding to $e$ in $\cS(H_n)$ is bounded by the horizontal lines with $y$-coordinates $y_n(u)$ and $y_n(v)$, and the vertical lines with $x$-coordinates $x_n(u^*)$ and $x_n(v^*)$. 
Since $y_n$ is harmonic with boundary $y_n(s) = 1$ and $y_n(t) = 0$, by Theorem~\ref{thm:harmconv}, $y_n$ converges pointwise. Furthermore, recall from (\ref{eq:cursum}) that 
\[
\lambda_n = \deg_n(s) -1 - \sum_{w \sim s} y_n(w)
\]
which in particular implies that $\lambda_n$ converges. 

Now let $\hat{x}_n:v(H_n^*) \to \mathbb{R}$ be harmonic on $H_n^*$ with boundary $\hat{x}_n(t^*) = 1$ and $\hat{x}_n(s^*) = 0$. By uniqueness and linearity of harmonic functions, $x_n = \lambda_n \hat{x}_n$. But $\hat{x}_n$ converges by Theorem~\ref{thm:harmconv}, so the same is true of $x_n$. It follows that square positions and sizes converge to their limiting values. For all $n$, the interiors of squares of $\cS(\rH_n)$ are pairwise disjoint; since the position and size of each square converges, the same must hold in $\cS(\rH_{\infty})$. 

By its definition, $\cS(H_n)$ is contained within $[0,\lambda_n] \times [0,1]$ for each $1 \le n \le \infty$. 
Since $\lambda_n \to \lambda_{\infty}$, the squarings $(\cS(\rH_n),1 \le n \le \infty)$ are uniformly bounded in $\mathbb{R}^2$. Furthermore, it is immediate from the energy formulation of resistance (see \citep[Proposition 9.2]{LyPe}) that $\lambda_\infty$ is precisely the sum of the areas of the squares $\{s_e, e \in e(H_{\infty})\}$.
Since these squares have disjoint interiors, they must therefore tile $[0,\lambda_{\infty}]\times[0,1]$; in other words, $\cS(\bG_\infty)$ is a squaring of $[0,\lambda_n]\times[0,1]$. It is then immediate that $\cS(\rH_n) \to \cS(\rH_{\infty})$ in the Hausdorff sense. 

We now allow that $\rH_n$ is infinite for $1 \le n < \infty$. View $\rH_n$ as a local weak limit of finite graphs; then the preceding case shows that $\cS(\rH_n)$ is a squaring of a rectangle for each $n$, and a reprise of the above arguments then shows that $\cS(\rH_n) \to \cS(\rH_{\infty})$ in the Hausdoff sense. 
\end{proof}
\begin{corollary}\label{prop:thm1part}
The squarings $S_n$ converge almost surely to a squaring $S_{\infty}$ as $n \to \infty$, for the Hausdorff distance, and $S_{\infty}=\cS(G_{\infty})$. Furthermore, $S_{\infty}$ has infinitely many squares of positive area. Finally, for all vertices $v$ (resp.\ faces $f$) of $G_{\infty}$, $\ell(v)$ (resp.\ $\ell(f)$) is non-degenerate. 
\end{corollary}
\begin{proof}
In view of Proposition~\ref{prop:conv_squaring}, the convergence is immediate from the a.s\ convergence of $\bG_n$ to $\bG_{\infty}$ described in Section~\ref{ssec:growth}, and the a.s.\ recurrence of $\bG_\infty$ from Theorem~\ref{thm:ginfinity_recur}. Next, Theorem~\ref{thm:distlim} implies that $G_{\infty}$ is a.s.\ $3$-connected. It follows by the same argument as for Lemma~\ref{lem:3con_nondegen} that the lines $\ell(v)$, $v \in v(G_{\infty})$ are all non-degenerate. On the other hand, a non-degenerate line $\ell(v)$ must neighbour a non-degenerate square $s_e$ (in fact, at least $3$ such squares since $G_{\infty}$ is $3$-connected, but we do not need this). Each square borders only two primal lines, so there must be an infinite number of non-degenerate squares. 
\end{proof}

\section{\large {\bf Only one point of accumulation}} \label{sec:onepoint}
To show that $S_{\infty}$ a.s.\ has only one point of accumulation, we use a result of He and Schramm \cite{he95hyperbolic}, which requires a brief introduction. 
A {\em packing} is a collection $\{P_i, i \in I\}$ of measurable subsets of $\mathbb{R}^2$ such that for each $i \in I$, the interior of $P_i$ is disjoint from $\bigcup_{j \in I\setminus \{i\}} P_j$. Its {\em contacts graph} is the graph $R=R(\{P_i, i \in I\})$ with vertices $\{P_i,i \in I\}$ and edges $\{\{P_i,P_j\}: P_i \cap P_j \ne \emptyset\}$. 

A measurable set $A \subset \mathbb{R}^2$ is {\em $\delta$-fat} if for all $x \in B$ and all $r > 0$ with $A \setminus B(x,r) \ne \emptyset$, $\mathrm{Leb}(A \cap B(x,r)) \ge \delta \mathrm{Leb}(B(x,r))$.\footnote{$\mathrm{Leb}$ denotes Lebesgue measure in $\mathbb{R}^2$.} To quote from \cite{he95hyperbolic}, ``a set is fat if if its area is roughly proportional to the square of its diameter, and this property also holds locally". 

A packing $\{P_i,i \in I\}$ is fat if there is $\delta > 0$ such that $P_i$ is $\delta$-fat for all $i \in I$. It is {\em well-separated} if for each $P_i$, the set $\bigcup_{P_j \sim P_i} P_j \setminus P_i$ contains a Jordan curve separating $P_i$ from $\bigcup_{P_j \nsim P_i} P_j \setminus P_i$. 

A graph $G$ is {\em one-ended} if for any finite set $U \subset v(G)$, $G-U$ has exactly one infinite connected component. 
It is {\em edge-parabolic} if there exists a function $m: e(G) \to [0,\infty)$ with $\sum_{e \in e(G)} m(e)^2 < \infty$ such that for any infinite path $\gamma$ in $G$, $\sum_{e \in e(\gamma)} m(e) = \infty$. Likewise, it is {\em vertex-parabolic} if there exists a function $m': v(G) \to [0,\infty)$ with $\sum_{v \in v(G)} m'(v)^2 < \infty$ such that for any infinite path $\gamma$ in $G$, $\sum_{v \in e(\gamma)} m'(v) = \infty$. For locally finite graphs, edge-parabolicity is equivalent to recurrence \cite{duffin62extremal,doyle84rw,he95hyperbolic}, and edge-parabolicity implies (but is not equivalent to) vertex-parabolicity \cite{he95hyperbolic}. 
%The following is a special case of \cite{he95hyperbolic}, Theorem~6.1. 
\begin{theorem}[\cite{he95hyperbolic}, Theorem~6.1]\label{thm:heschramm}
If $(P_i,i \in I)$ is a well-separated fat packing, and its contacts graph $R$ is locally finite, one-ended and vertex-parabolic, then $\bigcup_{i \in I} P_i$ has a single point of accumulation.
\end{theorem}
We abuse notation by writing $R(S_{\infty})$ for the contacts graph of the packing given by the squares of $S_{\infty}$. This packing is a.s.\ fat since all its bounded components are squares, and the aspect ratio of the rectangle they tile is a.s.\ finite. Since it tiles all of space, it is also easily seen to be well-separated. To conclude that $S_{\infty}$ has a single point of accumulation, it thus suffices to prove the following two propositions.

\begin{proposition}\label{prop:contacts_recurrent}
$R(S_{\infty})$ is almost surely vertex-parabolic.
\end{proposition}
\begin{proposition}\label{prop:contacts_oneend}
$R(S_{\infty})$ is almost surely one-ended.
\end{proposition}
We begin by proving Proposition~\ref{prop:contacts_oneend}, which is a consequence of the following straightforward fact. Slight modifications of Fact~\ref{fact:qoneend} have already appeared in the literature \cite{angel2003uipt,krikun05local}.
\begin{fact}\label{fact:qoneend}
$\bQ_{\infty}$ is almost surely one-ended. 
\end{fact}
\begin{proof}
First, $\bT_{\infty}$ is a.s. one-ended; this is well-known, but in particular follows from the explicit description of $\bT_{\infty}$ given in \cite{luczak04building} and reprised in \cite{ab14hex} . 
Next, for any finite set $U \subset v(\bQ_{\infty})$, if $\hat{U}$ is the pre-image of $U$ in $v(T_{\infty})$ then $\bT_{\infty}-\hat{U}$ has at least as many infinite connected components as $\bQ_{\infty}-U$; this is immediate from the fact that $\bT_{\infty}$ is formed from $\bT_{\infty}$ by making vertex identifications, since both $\bT_{\infty}$ and $\bQ_{\infty}$ are a.s.\ locally finite. Thus, for finite $U$, $\bQ_{\infty}-U$ a.s.\ has at most one infinite connected component; on the other hand, it a.s.\ has at least one such component since $\bQ_{\infty}$ is a.s.\ locally finite. 
\end{proof}

\begin{proof}[Proof of Proposition~\ref{prop:contacts_oneend}]%
We identify $\bQ_{\infty}$ with the set $v(G_{\infty}) \cup f(G_{\infty})$, where $f(G_{\infty})$ denotes the set of faces of $G_{\infty}$, so edges of $\bQ_{\infty}$ precisely encode incidences between vertices and faces in $G_{\infty}$. With this identification, for $u \in v(Q_{\infty})$ we write $\ell(u)$ for the horizontal or vertical line segment corresponding to $u$ as described at the end of Section~\ref{ssec:squaring}. 

Let $\{s_i,i \in I\}$ be a finite set of squares in $v(R(S_{\infty}))$ such that $\bigcup_{i \in I} s_i$ is simply connected. Then the set of squares in $v(R(S_{\infty}))\setminus \{s_i,i \in I\}$ incident to some square in $\{s_i,i \in I\}$ induces a connected subgraph of $R(S_{\infty})$, from which it is immediate that $R(S_{\infty}) - \{s_i,i \in I\}$ is (graph theoretically) connected. This implies that any finite separating set in $R(S_{\infty})$ contains a cycle in $R(S_{\infty})$. 

\begin{figure} 
\includegraphics[scale=1]{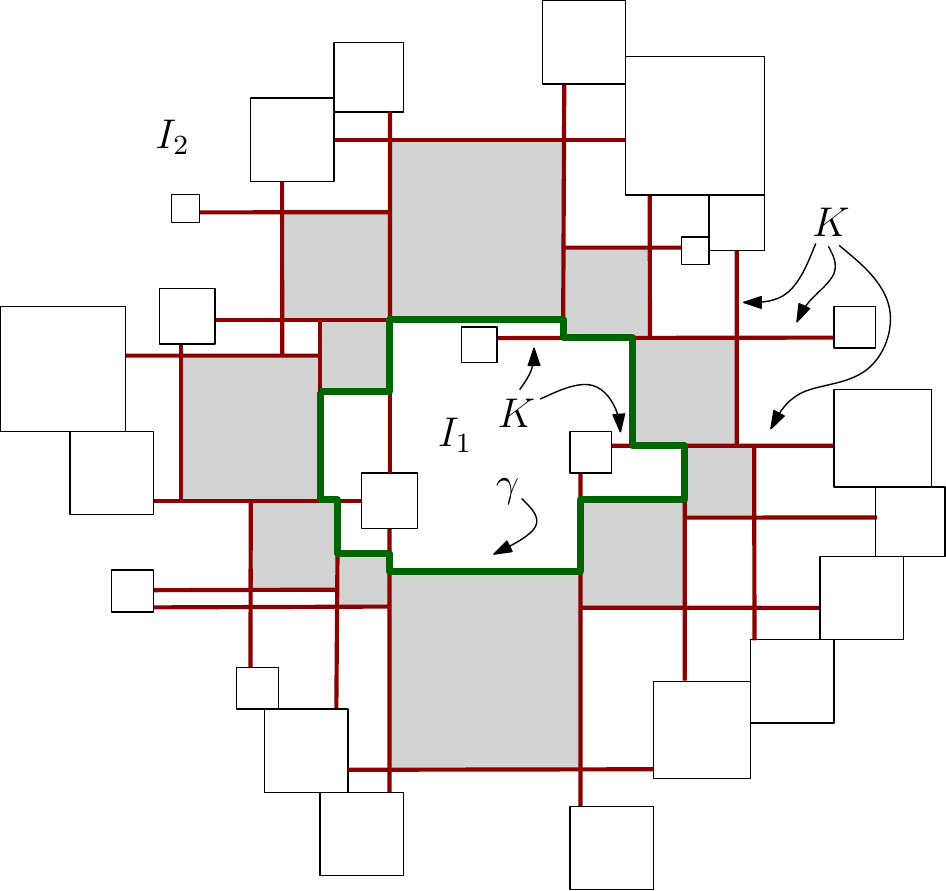}
\caption{\small The set $\{s_i,i \in I\}$
is given by the grey squares. Thicker lines (both red and green) are the lines $\ell(u)$ for $u \in K$. The thickest, dark-green, closed path is one possibility for $\gamma$.}
\label{fig:oneendfigure}
\end{figure}

Note that by Corollary~\ref{prop:thm1part} $R(S_{\infty})$ is a.s.\ an infinite graph. Now suppose $R(S_\infty)$ is not one-ended, let $\{s_i,i \in I\} \subset v(R(S_{\infty}))$ be such that $R(S_{\infty}) - \{s_i,i \in I\}$ contains at least two infinite connected components, 
and write $I_1$,$I_2$ for two such infinite components. 
Note that $\bigcup_{i \in I} \partial s_{i}$ contains a simple closed path $\gamma$ in $\mathbb{R}^2$ such that (relabelling $I_1$ and $I_2$ if necessary) all squares in $I_1$ (resp.\ $I_2$) lie in the interior (resp.\ exterior) of $\gamma$; see Figure~\ref{fig:oneendfigure}. Let $J_1$ (resp.\ $J_2$) be the set of vertices $u$ of $\bQ_{\infty}$ with $\ell(u)$ strictly contained in the interior (resp.\ exterior) of $\gamma$. It is clear that $\bQ_{\infty}[J_1]$ and $\bQ_{\infty}[J_2]$ each contain at least one infinite connected component. 

The set $\bigcup_{i \in I} \partial s_{i}$ consists of finitely many horizontal and vertical line segments; let 
\[
K = \left\{u \in v(\bQ_{\infty}):  \ell(u) \cap \bigcup_{i\in I} \partial s_i \neq \emptyset \right\} .
\]
Since $\bQ_{\infty}$ is locally finite and $S_{\infty}$ contains no degenerate lines, $K$ is a finite set. Now let $P=(u_1,u_2,\ldots,u_m)$ be any path in $\bQ_{\infty}$ joining $J_1$ and $J_2$. Then $\gamma \cap \bigcup_{i=1}^m \ell(u_i) \neq \emptyset$, so $P$ must contain a vertex from $K$. Therefore $\bQ_{\infty}-K$ has at least two infinite connected components, contradicting Fact~\ref{fact:qoneend}.
 \end{proof}

The remainder of the section is devoted to proving Proposition~\ref{prop:contacts_recurrent}.
It is natural to try a direct appeal to Theorem~\ref{thm:limrecur} to show that $R(S_{\infty})$ is recurrent. However, our information about the law of the sequence $(R(S_n),n \ge 1)$ seems too weak to apply this approach. More precisely, root $R(S_n)$ at its vertex which corresponds to the unbounded component of $S_n$. It is not clear how to show that the law of this vertex is (exactly or approximately) stationary conditional on $R(S_n)$.\footnote{One way to overcome this difficulty would be to derive more detailed enumerative information about the number of squarings with $n$ squares; question (9) of Section~\ref{sec:quests} relates to this.} Furthermore, we do not see an obvious choice of root vertex which would improve matters. 

Instead of applying Theorem~\ref{thm:limrecur} to the contact graphs, 
we we first define a sequence of graphs $(D_n,1 \le n \le \infty)$ to which Theorem~\ref{thm:limrecur} does apply to show that $D_{\infty}$ is recurrent and so edge-parabolic. We then show that edge-parabolicity of $D_{\infty}$ implies vertex-parabolicity for $R(S_{\infty})$.  
We now proceed to details. 

\begin{figure}[ht]
\begin{subfigure}[b]{0.3\textwidth}
		\includegraphics[width=\textwidth,page=7]{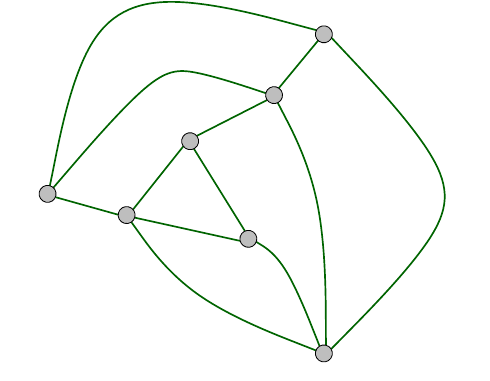}
               \caption{The map from Figure~\ref{fig:trivialbij}.}
                \label{fig:primal}
\end{subfigure}%
\quad
\begin{subfigure}[b]{0.3\textwidth}
		\includegraphics[width=\textwidth,page=8]{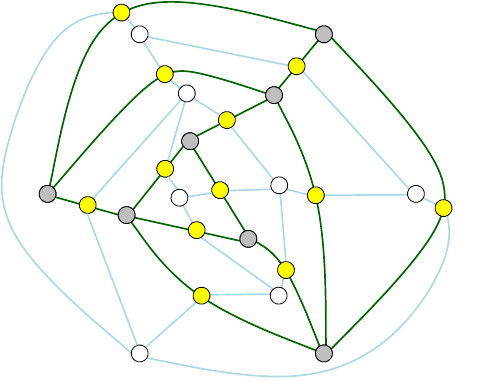}
                \caption{Its derived map}
                \label{fig:deriveddef}
\end{subfigure}%
\quad
\begin{subfigure}[b]{0.3\textwidth}
		\includegraphics[width=\textwidth,page=9]{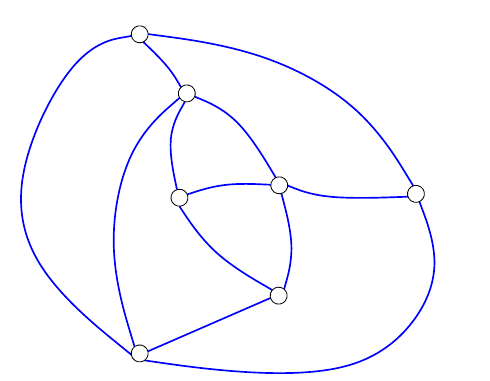}
                \caption{Its planar dual.}
                \label{fig:dual}
\end{subfigure}%
\quad
\caption{The derived map.} 
\label{fig:derived_def}
\end{figure}
Let $G$ be a planar map. The {\em derived map} of $G$, denoted $D=D(G)$ is obtained as follows. First, subdivide each edge of $G$ once; call the newly created vertices $\{v_e,e \in e(G)\}$. Second, add a vertex to the interior of each face, and join each facial vertex to all the incident subdivision vertices. See Figure~\ref{fig:derived_def}. Note that if $F$ is the planar dual of $G$ then $D(F)=D(G)$. 

\begin{lemma}\label{lem:derived_sub}
Let $\mathrm{G}=(G,st)$ be a planar map with squaring $S$. Then the contacts graph $R(S)$ is isomorphic to a subgraph of $D(G)^2$. 
\end{lemma} 
\begin{proof}
The vertices $v(R(S))$ correspond to squares of $S$, and thence to edges of $G$. This gives a natural map from $v(R(S))$ to the set of subdivision vertices $\{v_e,e \in e(G)\}$. We similarly associate primal vertices of $D(G)$ to vertices of $G$ and thence to primal lines of $S$, and facial vertices of $D(G)$ to faces of $G$ and thence to facial lines of $S$. (Primal and facial lines were defined in Section~\ref{ssec:squaring}.) 

Fix squares $s,\hat{s} \in  v(R(S))$ and write $e,\hat{e}$ for the corresponding edges of $G$.
 If $s$ and $\hat{s}$ border a common primal line then $e$ and $\hat{e}$ share a common endpoint, so $v_e$ and $v_{\hat{e}}$ are joined by a path of length two in $D(G)$; see Figure~\ref{fig:derived_a}. Since the derived graph of $G$ and of the dual of $G$ are identical, the same holds $s$ and $\hat{s}$ border a common facial line. To prove the lemma it thus suffices to show that if $v$ and $\hat{v}$ are adjacent in $R(S)$ then $s$ and $\hat{s}$ border a common primal or facial line. This is obvious unless $s$ and $s'$ meet at a single point $x \in \mathbb{R}^2$. 

If the latter occurs then there are precisely $4$ squares that meet at $x$. If two primal lines meet at $x$ then the corresponding vertices of $G$ lie on a common face, and the associated facial line passes through $x$ and thus borders both $s$ and $\hat{s}$ (see Figure~\ref{fig:derived_c}). 
Otherwise, a primal line passes through $x$ and thus borders both $s$ and $\hat{s}$. In either case $s$ and $\hat{s}$ border a common primal or facial line (see Figure~\ref{fig:derived_b}). This completes the proof. 
\end{proof}
\begin{figure}[ht]
\begin{subfigure}[b]{0.3\textwidth}
		\includegraphics[width=\textwidth,page=1]{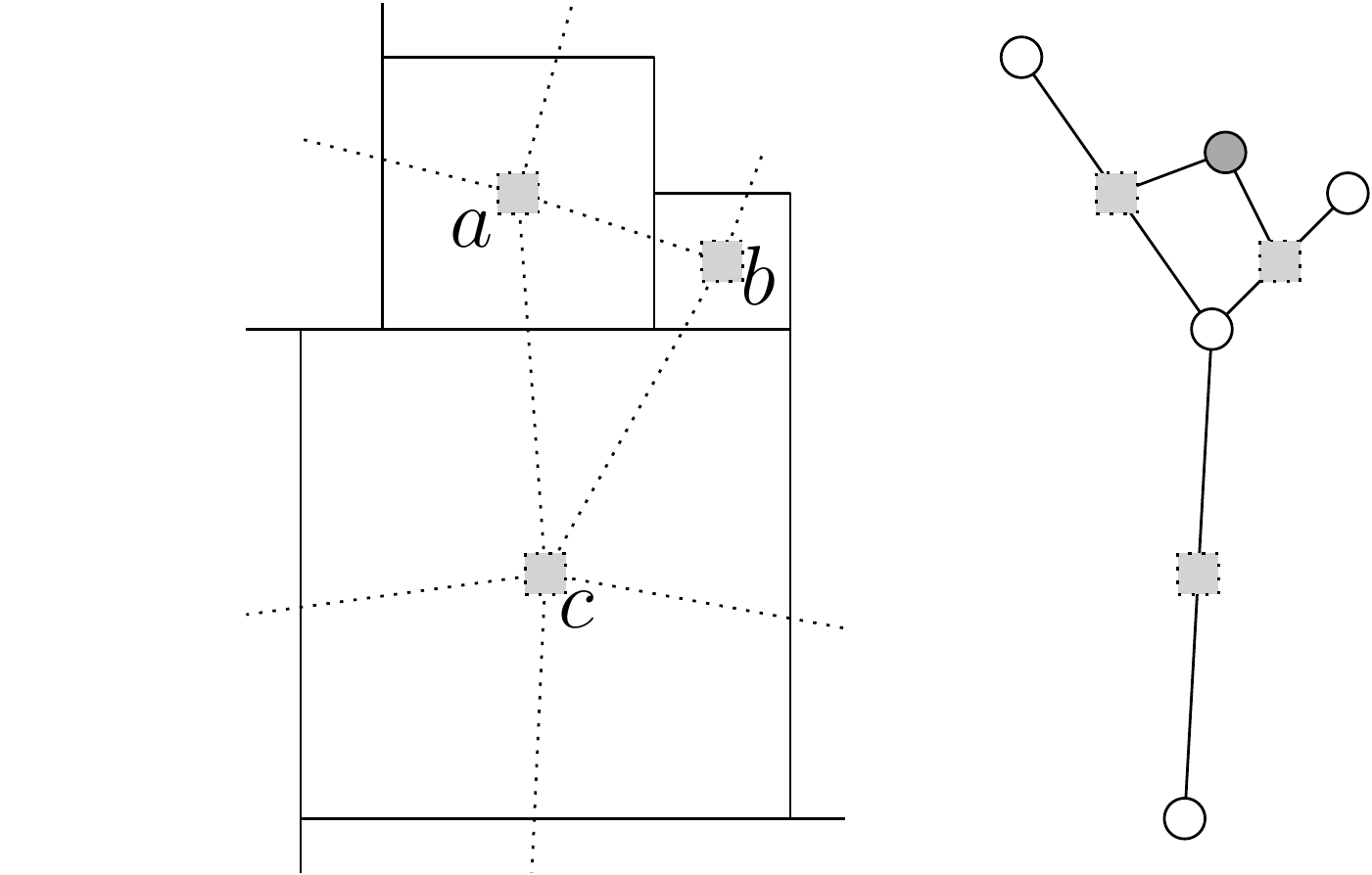}
                \caption{$a,b,c$ are vertices of the contacts graph.}
                \label{fig:derived_a}
\end{subfigure}%
\quad
\begin{subfigure}[b]{0.3\textwidth}
		\includegraphics[width=\textwidth,page=3]{figures/derived.pdf}
                \caption{The bold vertical line is (part of) a facial line}
                \label{fig:derived_c}
\end{subfigure}%
\quad
\begin{subfigure}[b]{0.3\textwidth}
		\includegraphics[width=\textwidth,page=2]{figures/derived.pdf}
                \caption{The bold horizontal line is (part of) a primal line.}
                \label{fig:derived_b}
\end{subfigure}%
\quad
\caption{\small Vertices of the contacts graph, and the corresponding vertices of the derived map, are drawn as squares. Primal vertices are white disks, and dual vertices are grey disks. \\(A) If two squares border a common primal (resp. facial) line then their corresponding vertices in the derived graph are joined by a path of length $2$ through a primal (resp.~dual) vertex. (B) and (C): if squares meet at a point then they border a common primal or facial line.} 
\label{fig:derived}
\end{figure}

For $1 \le n \le \infty$, write $D_n$ for the derived map of $G_n$. 
\begin{lemma}\label{lem:derived_recurrent}
$D_{\infty}$ is almost surely recurrent. 
\end{lemma} 
\begin{proof}
Let $\hat{\bG}_n=(\hat{G}_n,\hat{s}_n\hat{t}_n)$ be as in Theorem~\ref{thm:distlim}. Let $\hat{D}_n$ be the derived map of $\hat{G}_n$, let $\hat{v}_n$ be the subdivision vertex corresponding to $\{\hat{s}_n,\hat{t}_n\}$, and let $\rho_n=\hat{s}_n$ with probability $1/2$ and $\rho_n=\hat{v}_n$ with probability $1/2$. Likewise, let $v_{st}$ be the subdivision vertex corresponding to edge $\{s,t\}$ in $G_{\infty}$, and let $\rho_{\infty}$ be either $s$ or $v_{st}$, each with probability $1/2$. Then since $\hat{\bG}_n$ converges in distribution to $\bG_{\infty}$, it also holds that $(\hat{D}_n,\rho_n)$ converges in distribution to $(D_{\infty},\rho_{\infty})$. 

Next, since $\hat{s}_n\hat{t}_n$ is a uniformly random edge of $\hat{G}_n$, it is immediate from the definition of $\rho_n$ that $\rho_n$ is distributed according to the stationary law of $\hat{D}_n$. It follows that $(D_{\infty},\rho_{\infty})$ is a distributional limit of finite planar graphs. Furthermore, if $\rho=s$ then the degree $\deg_{D_{\infty}}(\rho)=\deg_{G_{\infty}}(s)$, and otherwise $\deg_{D_{\infty}}(\rho)=4$. By Fact~\ref{fact:expvert} it follows that $\deg_{D_{\infty}}(\rho')$ has exponential tails, and thus by Theorem~\ref{thm:limrecur}, $D_{\infty}$ is almost surely recurrent. 
\end{proof}

\begin{proof}[Proof of Proposition~\ref{prop:contacts_recurrent}]
By Lemma~\ref{lem:derived_recurrent}, $D_{\infty}$ is almost surely recurrent and so edge-parabolic. Let $m: e(D_{\infty}) \to [0,\infty)$ be such that $\sum_{e\in e(D_{\infty})} m(e)^2 < \infty$ and such that $\sum_{e \in e(\gamma)} m(e)$ is infinite for any infinite path $\gamma$ in $D_{\infty}$. 
Use $m$ to define a function $m': v(R_{\infty}) \to [0,\infty)$ as follows. For $s_e \in v(R_{\infty})$, let $v_e$ be the corresponding subdivision vertex of the derived graph, and let 
\[
m'(s_e) = \sum_{\{f \in e(D_{\infty}): v_e \in f\}} m(f). 
\]
In other words, $m'(s_e)$ is the sum of the $m$-masses of the four edges incident to $v_e$ in the derived graph.

Each subdivision vertex has degree four in the derived graph, and each edge of the derived graph is incident to exactly one subdivision vertex. It follows by Cauchy-Schwarz that 
\begin{align*}
\sum_{s \in v(R(S_{\infty}))} m'(s)^2 & = \sum_{e \in e(G_{\infty})} \left(\sum_{\{f \in e(D_{\infty}): v_e \in f\}} m(f)\right)^2   \\
		& \le \sum_{e \in e(G_{\infty})} 4 \sum_{\{f \in e(D_{\infty}): v_e \in f\}} m(f)^2 \\
		& = \sum_{f \in e(D_{\infty})} 4m(f)^2  < \infty. 
\end{align*}

Now suppose that $\gamma$ is an infinite path in the contacts graph. Then Lemma~\ref{lem:derived_sub} implies that 
$
\bigcup_{s_e \in v(\gamma)} \{f \in e(D_{\infty}): v_e \in f\} 
$
contains an infinite path $\gamma'$ in $D_{\infty}$. 
It follows that 
\[
\sum_{s_e \in v(\gamma)} m'(s_e) = \sum_{s_e \in v(\gamma)} \sum_{\{f \in e(D_{\infty}): v_e \in f\}} m(f) \ge 
\sum_{e \in e(\gamma')} m(e) = \infty, 
\]  
the last equality by our choice of $m$. Since $\gamma$ was an arbitrary infinite path, it follows that $R(S_{\infty})$ is vertex-parabolic. 
\end{proof}

\section{\large {\bf Proof of Theorem~\ref{thm:main}}} \label{sec:proof}
The a.s.\ Hausdorff convergence of $S_n$ to $S_{\infty}=\cS(\bG_{\infty})$ was established in Corollary~\ref{prop:thm1part}; it remains to show that $S_{\infty}$ a.s.\ has exactly one point of accumulation. 
 
Since $S_{\infty}$ is an infinite squaring and is compact, it clearly has at least one point of accumulation. 
The fact that $S_{\infty}$ has at most one point of accumulation follows from Theorem~\ref{thm:heschramm}, once it is verified that the contacts graph $R(S_{\infty})$ is one-ended and vertex-parabolic; this was accomplished in Propositions~\ref{prop:contacts_recurrent} and~\ref{prop:contacts_oneend}. \qed

\section{\large {\bf Further Questions and Topics}} \label{sec:quests}

\begin{enumerate}
\item We begin with an analogue of Conjecture 7.1 of \cite{duplantier2011liouville} and of Conjecture 1 (a) of \cite{sheffield10conformal}, for the random squarings $S_n$. There is a unique translation and scaling under which the image $S_n'$ of $S_n$ is centred at 0 and such that when $S_n'$ is stereographically projected to the Riemann sphere $\mathbb{C}^*=\mathbb{C} \cup \{\infty\}$, the image of the unbounded region of $\mathbb{R}^2\setminus S_n$ has area $1/n$. Apply this transformation, and let $\mu_n$ be the measure on $\mathbb{C}^*$ obtained by letting each connected component of $\mathbb{C}^*\setminus S_n'$ have measure $1/n$.\footnote{The measure $\mu_n$ is uniquely determined if we also specify that its restriction to any component of $\mathbb{C}^*\setminus S_n'$ is a multiple of the surface measure of the Riemann sphere.} Then $\mu_n$ should converge weakly to a measure $\mu$ on $\mathbb{C}^*$ which is some version of the Liouville quantum gravity measure (possibly the ``$\gamma$-unit area quantum sphere measure with $\gamma=\sqrt{8/3}$", introduced in \cite{sheffield10conformal}). In particular, $\mu$ should satisfy a version of the KPZ dimensional scaling relation. 
\item We expect that the box-counting dimension of $S_{\infty}$ is a.s. well-defined and constant. More precisely, write $n_{\epsilon}$ for the number of balls of radius $\epsilon$ required to cover $S_{\infty}$. We expect that $\log n_{\epsilon}/\log (1/\epsilon) \to c$ almost surely, where $c$ is non-random.  Is this true? If so, what is $c$? Is $c > 1$? (Note that for the Hausdorff dimension, if $(C_n,n \in \mathbb{N})$ are measurable sets in $\mathbb{R}^2$ then $\mathrm{dim}_{\mathrm{Haus}}(\bigcup_{n \in \mathbb{N}} C_n) = \sup_{n \in \mathbb{N}} \mathrm{dim}_{\mathrm{Haus}}(C_n)$. Since $S_{\infty}$ is a countable union of line segments, it follows that $\mathrm{dim}_{\mathrm{Haus}}(S_{\infty}) = 1$ almost surely.)
\item Let $Z$ be the a.s.\ unique accumulation point of $S_{\infty}$. Can the law of $Z$ be explicitly described?
\item Write $G_{\infty}(\epsilon)$ for the graph induced by those vertices for which all incident squares are disjoint from $B(Z,\epsilon)$. How quickly does $G_{\infty}(\epsilon)$ grow as $\epsilon$ decreases? Relatedly, how does the diameter of $G_{\infty}(\epsilon)$ grow? Existing results about random maps suggest that if the diameter grows as $\epsilon^{-\alpha}$ then the volume should grow as $\epsilon^{-4\alpha}$. 
\item The structure of $S_{\infty}$ near $Z$ should be independent of its structure near the root; here is one question along these lines. Reroot $G_{\infty}$ by taking one step along a random walk path from the root, write $\hat{S}_{\infty}$ for the resulting squaring and $\hat{Z}$ for its point of accumulation. Then recenter $S_{\infty}$ and $\hat{S}_{\infty}$ so that $Z$ and $\hat{Z}$ sit at the origin. Does $\epsilon^{-1}d_H(S_{\infty}\cap B(0,\epsilon),\hat{S}_{\infty}\cap B(0,\epsilon)) \to 0$ almost surely, as $\epsilon \to 0$? Here $d_H$ denotes Hausdorff distance. 
\item Let $e_n(1),\ldots,e_n(k)$ be independent, uniformly random oriented edges of the contacts graph $R(S_n)$, and for $1 \le i\le k$ let $r_n(i)$ be the ratio of the side length of the ``tail square" of $e_n(i)$ to that of its ``head square". The vector $(r_n(i),1\le i \le k)$ should converge in distribution to a limit $(r(i),1 \le i \le k)$, whose entries are iid. This would be a very small first step towards establishing that the random squaring in some sense ``looks like the exponential of a Gaussian free field". 
\item Let $A_n$ be the adjacency matrix of $G_n$. The areas of squares may be calculated as determinants of minors of $A_n$. However, these determinants grow very quickly, and even finding logarithmic asymptotics seems challenging. A simpler, still challenging project is to study the determinant of any principal minor of $A_n$ or, equivalently, to study the number of spanning trees of $G_n$. 
\item The height of $S_{\infty}$ is $1$ but its width $W_{\infty}$ is random, and by considering the graph structure near the root of $G_{\infty}$ it is not hard to see that $W_{\infty}$ is an honest random variable (rather than a.s.\ constant) On the other hand, duality implies that $W_{\infty}$ and $1/W_{\infty}$ have the same law. Can anything explicit be said about this law? In particular, is $\mathbb{P}(W_{\infty}=1)>0$?
\item Simulations suggest that for $n$ large, $S_{n}$ is unlikely to contain four squares with common intersection. Does this probability indeed tend to zero as $n$ becomes large? This question looks innocent. However, recall that such intersections are the reason the function sending a rooted planar graph to its squaring is non-invertible. A positive answer would constitute substantial progress towards proving an asymptotic formula, conjectured by Tutte \citep[Section 9]{tutte63census}, for the number of perfect squarings with $n$ squares. 
\item Let $\hat{S}_n$ be uniformly distributed over squarings of a rectangle with $n$ squares. Does $\hat{S}_n$ converge in distribution to $S_{\infty}$ for the Hausdorff distance? This follows if the laws of $S_n$ and $\hat{S}_n$ are close, which would itself follow from a positive answer to the previous question. 
\item The behaviour of the simple random walk on $G_{\infty}$ is also of interest. How do quantities such as $\mathbb{P}(X_t=X_0)$, $d_{G_{\infty}}(X_0,X_t)$, and $\#\{X_s,0 \le s \le t\}$ scale in $t$? 
\item It seems likely that $R(S_{\infty})$ is recurrent; is it? Here is one tempting argument for recurrence; its incorrectness was pointed out to us by Ori Gurel-Gurevich. By Lemma~\ref{lem:derived_sub}, $R(S_{\infty})$ may be viewed as a subgraph of $D_{\infty}^2$. Since $D_{\infty}$ is recurrent, so is $D_{\infty}^2$; then conclude via Rayleigh monotonicity. The problem with the argument is that the recurrence of $D_{\infty}$ is not known to imply the recurrence of $D_{\infty}^2$ (this implication would be true if $D_{\infty}$ had uniformly bounded degrees \citep[Theorem 2.16]{LyPe}). Perhaps if $G$ is a recurrent, unimodular random graph whose root degree has exponential tail, then any finite power of $G$ is also recurrent; this would be an interesting fact in its own right. 
\end{enumerate}

\subsection*{Acknowledgements}
Both authors thank Gr\'egory Miermont for useful comments, and Ori Gurel-Gurevich for pointing out an error in an earlier version of this work. LAB thanks Nicolas Curien and Omer Angel for their insightful remarks subsequent to a presentation of this work at the McGill Bellairs Research Institute, in April 2014.

\end{document}